\documentclass[10pt,a4paper, xcolor]{article}
\usepackage{amsfonts}
\usepackage{amssymb,amsthm, amsmath,graphicx,bm,amsthm,mathtools }
\usepackage{mathrsfs}
\usepackage{color,enumerate,lineno}
\usepackage{url,hyperref}
\usepackage{color,enumerate}
\usepackage{url,lineno}
\usepackage[latin1]{inputenc}
\usepackage[all]{xy}
\usepackage{algorithm} 
\usepackage{algpseudocode}

\newtheorem{theorem}{Theorem}[section]

\newtheorem{proposition}[theorem]{Proposition}
\newtheorem{corollary}[theorem]{Corollary}
\newtheorem{lemma}[theorem]{Lemma}

\theoremstyle{remark}

\newtheorem{example}[theorem]{Example}
\newtheorem{remark}[theorem]{Remark}

\def\c{{\mathrm c}}

\def\e{{\mathrm e}}
\def\g{{\mathrm g}}

\def\m{{\mathrm m}}

\def\t{{\mathrm t}}

\def\B{{\mathrm B}}
\def\C{{\mathrm{C}}}
\def\F{{\mathrm F}}
\def\G{{\mathrm G}}

\def\PF{{\mathrm {PF}}}

\def\Ap{{\mathrm{ Ap}}}
\def\max{{\mathrm{ max}}}

\def\gcd{{\mathrm{gcd}}}

\def\max{{\mathrm{max}}}
\def\msg{{\mathrm{ msg }}}
\def\mod{{\mathrm{ mod }}}

\def\CF{\mathcal{F}}

\def\CC{\mathscr{C}}

\def\N{\mathbb{N}}

\def\Z{\mathbb{Z}}

\def\rank{\mathrm{rank}\, }
\def\Ap{\mathrm{Ap}}
\def\SG{\mathrm{SG}}

\def\int{\mathrm{int}}

\title{Semi-covariety of numerical semigroups}

\author{
M. A. Moreno-Fr\'{\i}as \footnote{
    Dpto. de Matem\'aticas, Facultad de Ciencias,
    Universidad de C\'adiz, E-11510, Puerto Real  (C\'{a}diz, Spain).
    Partially supported by  Junta de Andaluc\'{\i}a group FQM-298 and by ProyExcel-00868.
     E-mail: mariangeles.moreno@uca.es.}
\and
 J. C. Rosales \footnote{
    Dpto. de \'Algebra, Facultad de Ciencias, Universidad de Granada,
    E-18071, Granada. (Spain).
    Partially supported by  Junta de Andaluc\'{\i}a group FQM-343 and by ProyExcel-00868.
    E-mail: jrosales@ugr.es.
}
 }

\date{}

\begin{document}

\maketitle

\begin{abstract}

	The main aim of this work is to introduce and  justify the study of  semi-covarities. 
	 A {\it semi-covariety}  is a non-empty family $\CF$ of numerical semigroups such   that it is closed under finite intersections, has a minimum, $\min(\CF),$ and if $S\in  \CF$ being  $S\neq \min(\CF),$ then there is  $x\in S$ such that $S\backslash \{x\}\in  \CF.$ As examples, we will  study the semi-covariety formed by all the numerical semigroups containing a fixed numerical semigroup, as well as,  the semi-covariety composed by 
	  all the numerical semigroups of coated odd elements and fixed Frobenius number.


    \smallskip
    {\small \emph{MSC-class:} 20M14, 11D07 }
\end{abstract}

\section{Introduction}

Let $\Z$ be the set of integers and $\N=\{z\in \Z \mid z\ge 0\}$. A {\it submonoid} of $(\N,+)$ is a subset  of $\N$ which is closed under addition and contains the element $0.$ A {\it numerical semigroup} is a
submonoid $S$ of $(\N,+)$ such that $\N\backslash
S=\{x\in \N \mid x \notin S\}$ has finitely many elements.

If $S$ is a numerical semigroup, then $\m(S)=\min(S\backslash \{0\})$, $\F(S)=\max\{z\in \Z \mid z \notin S\}$ and $\g(S)=\#(\N \backslash S)$ (where $\# X $ denotes the cardinality of a set $X$) are 
three important invariants of $S,$  called   {\it multiplicity},  {\it Frobenius number} and  {\it genus} of $S$, respectively.

If $A$ is a subset nonempty of $\N$, we denote by $\langle A
\rangle$ the submonoid of $(\N,+)$ generated by $A$, that is,
$\langle A \rangle=\{\lambda_1a_1+\dots+\lambda_na_n \mid n\in
\N\backslash \{0\}, \, \{a_1,\dots, a_n\}\subseteq A \mbox{ and }
\{\lambda_1,\dots,\lambda_n\}\subseteq \N\}.$  In \cite[Lema 2.1]{libro}) it is shown that $ \langle A \rangle$
is a numerical semigroup if and only if $\gcd(A)=1.$ 

If $M$ is a submonoid of $(\N,+)$ and $M=\langle A \rangle$, then we
say that $A$ is a {\it system of generators} of $M$. Moreover, if $M\neq
\langle B \rangle$ for all $B \varsubsetneq A$, then we will say
that $A$ is a {\it minimal system of generators} of $M$. In
\cite[Corollary 2.8]{libro} is shown that every submonoid of
$(\N,+)$ has a unique minimal system of generators, which in
addition is finite. We denote by $\msg(M)$ the minimal system of
generators of $M$. The cardinality of $\msg(M)$ is called the {\it
	embedding dimension }of $M$ and will be denoted by $\e(M).$

The Frobenius problem (see \cite{alfonsin})  focuses on finding formulas to calculate the Frobenius number and the genus of a numerical semigroup from its minimal  system of generators. The problem was solved in \cite{sylvester} for numerical semigroups with embedding dimension two.  Nowadays, the problem is still open in the case of numerical semigroups with embedding dimension  greater than or equal to three. Furthemore, in this case the problem of computing the Frobenius number of a general numerical semigroup becomes NP-hard.

 A {\it semi-covariety}  is a nonempty family $\CF$ of numerical semigroups that  fulfills  the following conditions:
\begin{enumerate}
	\item[1)]  $\CF$ has a minimum, with respect to  inclusion, denoted by $\min(\CF).$
	\item[2)] If $\{S, T\} \subseteq \CF$, then $S \cap  T \in \CF$.
	\item[3)]  If $S \in \CF$  and $S \neq  \min(\CF)$, then there is $x\in \msg(S)$ such that $S \backslash \{x\} \in \CF$.
\end{enumerate}

The concept  of semi-covariety generalizes the concepts of covariety and ratio-covariety introduced in \cite{covariedades} and \cite{ratio}, respectively.

In Section 2, we will see that every semi-covariety is finite and its elements can  be arranged in a tree.  Moreover, we present a characterization of the children of an arbitrary vertex in this tree. This fact will allow us to give an algorithmic procedure  to calculate all the elements of a semi-covariety.

Let $\CF$ be a semi-covariety. A set $X$ is an $\CF$-set if $X\cap \min(\CF)=\emptyset$ and $X \subseteq S$ for some $S\in \CF.$ In Section 2, we will see that if $X$ is an $\CF$-set, then there is the smallest element of $\CF$ that contains $X.$ This element will be denoted by $\CF[X].$ The $\CF$-{\it rank} of $S\in \CF$  is $\CF\rank(S)=\min\{\ \# X\mid X \mbox{ is a }\CF\mbox{-set and }\CF[X]=S \}.$ 

Let $\Delta$ be a numerical semigroup, in Section 3 we will show that $\theta(\Delta)=\{S\mid  S \mbox{ is a numerical semigroup and  }S\subseteq \Delta\}$ is a semi-covariety. The results of Section 2, allows us to present an algorithm which computes all the elements of $\theta(\Delta).$ We show that if $X$ is a $\theta(\Delta)$-set, then $\theta(\Delta)[X]=\langle X \rangle + \Delta.$ Besides, we study the elements of $\theta(\Delta)$ with $\theta(\Delta)$-rank 1.

Following the notation introduced in \cite{coe}, a {\it numerical semigroup of coated odd elements} (hereinafter, {\it Coe-semigroup}) is a numerical semigroup $S$ verifying that $\{x-1,x,x+1\}\subseteq S$ if $x\in S$ is odd. In \cite[Proposition 2]{coe} it is shown that if $S$ is a coe-semigroup, then $\F(S)$ is odd.

If $F$ is a positive odd integer, then we denote by $\CC(F)=\{S \mid S \mbox{ is a coe-semi-}\\ \mbox{group and }\F(S)=F\}.$ In section 4, we will see that $\CC(F)$ is a semi-covariety. The results of Section 2, will enable us to show an algorithm to compute all elements of $\CC(F).$\\

If $x$ is a positive odd integer, then we denotre $\c(x)=\{x-1,x+1\}.$  If $X \subseteq \N \backslash \{0\},$ then $\C(X)=X\cup \left ( \displaystyle \bigcup_{ \begin{array}{c}
		x\in X\\
		x \mbox{ impar}\\
		\end{array}}
		\c(x)
		\right).$ In Section 4, we will prove that if $X$ is a $\CC(F)$-set, then $\CC(F)[X]=\langle \C(X)\rangle \cup \{F+1,\rightarrow\}$ 
		 (the symbol $\rightarrow$ means that every integer greater than $F+1$
		belongs to the set). Finally, we will study the elements of $\CC(F)$ with $\CC(F)$-$\rank$1.
		
		\section{Semi-covarieties}
		
		Throughout this section, $\CF$ will denote a semi-covariety. 
		If $T$ is a numerical semigroup, then $\N\backslash T$ is a finite set. As a consequence, we deduce the following result.
		\begin{lemma}\label{lema1} If $T$ is a numerical semigroup, then the set $\{S\mid S \mbox{ is a numerical}\\ \mbox{semigroup and } T\subseteq S\}$ is finite. 
		\end{lemma}
		
		As $\CF \subseteq \{S\mid S \mbox{ is a numerical semigroup and }\min(\CF)\subseteq S\},$ then by applying Lemma \ref{lema1}, we obtain the next result.
		\begin{proposition}\label{proposition2}
		$\CF$ is a finite set. 
		\end{proposition}
		Let $S\in \CF\backslash \{\min(\CF)\}$ and let $$ \mu(\CF, S)=\min\{x\in \msg(S)\mid S \backslash \{x\}\in \CF\}.$$ 
		
		 Define recursively the following sequence of elements from  $\CF$:
		\begin{itemize}
			\item $S_0=S$,
			\item $S_{n+1} =\left\{\begin{array}{ll}
				S_n\backslash \{\mu(\CF,S_n)\}  & \mbox{if }   S_n\neq \min(\CF),\\
				\min(\CF) & \mbox{otherwise.}
			\end{array}
			\right.$
		\end{itemize}
		
		The sequence $\{S_n\}_{n\in \N}$ is called the $\CF$-{\it sequence associated} to $S.$
		The following result has  an immediate proof. 
		
		\begin{lemma}\label{lemma3}
			If $S \in \CF$ and $\{S_n\}_{n\in \N}$ is the  $\CF$-sequence associated to $S,$ then  there is $k\in \N$ such that $\min(\CF)=S_k\subsetneq S_{k-1} \subsetneq \dots \subsetneq S_0=S.$ Moreover, the cardinality of $S_i\backslash S_{i+1}$ is equal to $1$ for every $i\in \{0,1,\dots,k-1\}.$ 	
		\end{lemma}
		
		A {\it graph} $G$ is a pair $(V,E)$ where $V$ is a non-empty set and
		$E$ is a subset of $\{(u,v)\in V\times V \mid u\neq v\}$. The
		elements of $V$ and $E$ are called {\it vertices} and {\it edges},
		respectively. A {\it path (of
			length $n$)} connecting the vertices $x$ and $y$ of $G$ is a
		sequence of different edges of the form $(v_0,v_1),
		(v_1,v_2),\ldots,(v_{n-1},v_n)$ such that $v_0=x$ and $v_n=y$.

		A graph $G$ is {\it a tree} if there exists a vertex $r$ (known as
		{\it the root} of $G$) such that for any other vertex $x$ of $G,$
		there exists a unique path connecting $x$ and $r$. If  $(u,v)$ is an
		edge of the tree $G$, we say that $u$ is a {\it child} of $v$.
		
		Define the graph $\G(\CF)$ in the following way:
		\begin{itemize}
			\item $\CF$  is the set of vertices. 
			\item $(S,T)\in \CF \times \CF$ is an edge  if and only if $T=S\backslash \{\mu(\CF,S)\}.$
		\end{itemize}
		
		As a consequence of Lemma \ref{lemma3}, we have the following result.
		\begin{proposition}\label{proposition4}
			 $\G(\CF)$ is a tree with root $\min(\CF).$	
		\end{proposition}
		
		A tree can be  built recurrently starting from the root  and connecting, 
		through an edge, the vertices already built with  their  children. Hence, it is very interesting to characterize the children of an arbitrary vertex in $\G(\CF).$
		
		Following the terminology introduced in \cite{JPAA}, an integer $z$ is a {\it pseudo-Frobenius number} of $S$ if $z\notin S$ and $z+s\in S$ for all $s\in S\backslash \{0\}.$  We denote by $\PF(S)$
		the set of pseudo-Frobenius numbers of $S.$ The cardinality of $\PF(S)$ is an important invariant of $S$ (see \cite{froberg} and \cite{barucci}) called the {\it type} of $S,$  denoted by $\t(S).$
		
		If $S$ is a numerical semigroup,  denote by $\SG(S)=\{x\in \PF(S)\mid 2x\in S\}.$ Its elements will be called the {\it special gaps }of $S.$ The following result is Proposition 4.33 from \cite{libro}.
		
		\begin{proposition}\label{proposition5} Let $S$ be a numerical semigroup and $x\in \N\backslash S.$ Then $x\in \SG(S)$ if and only if $S\cup \{x\}$ is a numerical semigroup.
		\end{proposition}
		\begin{proposition}\label{proposition6}
		If $S\in \CF,$ then the set formed by the children of $S$ in the tree $\G(\CF)$ is  
		$$
		\{S\cup \{x\}\mid x\in \SG(S), S\cup \{x\}\in \CF \mbox{ and }\mu(\CF,S\cup \{x\})=x\}.
		$$
		\end{proposition}
		\begin{proof}
\begin{itemize}
	\item If $T$ is a child of $S,$ then $T\in \CF$ and $T\backslash \{\mu(\CF,T)\}=S.$ Therefore, $S\cup \{\mu(\CF,T)\}=T,$ $\mu(\CF,T)\in \SG(S),$ $S\cup \{\mu(\CF,T)\}\in \CF$ and $\mu(\CF,S\cup\{\mu(\CF,T)\})=\mu(\CF,T).$
	\item As $\mu(\CF,S\cup \{x\})=x,$ then $(S\cup \{x\})\backslash \{\mu(\CF,S\cup \{x\})\}=S$ and so $S\cup \{x\}$ is a child of $S.$
\end{itemize}			
		\end{proof}
	\begin{algorithm}\label{algorithm7}
		\caption{Computation of $\CF$} \mbox{}\par
%
		\begin{algorithmic}[1]
			\item[(1)] $\CF=\{\min(\CF)\},$ $\B=\{\min(\CF)\}.$
			\item[(2)] For all $S\in B$ compute $H(S)=\{T\in \CF \mid T \mbox{ is a child of }S\}.$
			\item[(3)] $C=\displaystyle \bigcup_{S\in B}H(S).$
			\item[(4)] If $C=\emptyset,$ then return $\CF.$
			\item[(5)] $\CF:=\CF \cup C,$ $B:=C$ and go to Step (2).
		\end{algorithmic} 
	\end{algorithm}

Recall that an $\CF$-set is a set $X$ verifying $X \cap \min(\CF)=\emptyset$ and $X\subseteq S$ for some $S\in \CF.$		

 If $X$ is an $\CF$-set, then we denote by $\CF[X]$ the intersection of all 
the elements of  $\CF$ containing $X.$ As $\CF$ is finite, then the intersection of elements from $\CF$ is also an element of  $\CF$. Hence, we have  the following result. 

\begin{proposition}\label{proposition8}
	If $X$ is an $\CF$-set, then $\CF[X]$ is the least element (with respect to  inclusion)  of  $\CF$ containing $X.$
\end{proposition}

If $X$ is an $\CF$-set and $S=\CF[X]$, then we say that $X$ is an $\CF$-{\it system of generators} of $S.$ Furthemore, if $S\neq
\CF[Y]$ for all $Y \varsubsetneq X$, then we will say
that $X$ is an $\CF$-{\it minimal system of generators} of $S.$

\begin{remark}\label{renark9}
	In \cite[Example 4]{covariedades} it is shown a covariety  $\CC$ and an element $S\in \CC$ such that $S$ admits two $\CC$-minimal system of generators. As every covariety is a semi-covariey, then we can assert that, in general, the $\CF$-minimal system of generators of an element $S\in \CF$ is not unique.
\end{remark}

\begin{proposition}\label{proposition10} If $S\in \CF,$ then $A=\{x\in \msg(S)\mid x\notin \min(\CF)\}$ is an $\CF$-set and $\CF[A]=S.$
	\end{proposition}
	\begin{proof}
		It is clear that $A$ is an $\CF$-set. As $S\in \CF$ and $A\subseteq S,$ then by Proposition \ref{proposition8}, $\CF[A]\subseteq S.$ Let us now look at the reverse inclusion. If $T\in \CF$ and $A\subseteq T,$ then $A\cup \min(\CF)\subseteq T.$ Thus, $\msg(S)\subseteq T$ and so $S\subseteq T.$  By applying  again  Proposition \ref{proposition8},  we have that $S\subseteq \CF[A].$ 
	\end{proof}
	
If $S\in \CF,$ then   the $\CF$-{\it rank} of $S$  is $\CF\rank(S)=\min\{\ \# X\mid X \mbox{ is an }\CF\mbox{-set and }S=\CF[X] \}.$

As an immediate consequence of Proposition \ref{proposition10}, we obtain the next result.
\begin{corollary}\label{corollary11}
	If $S\in \CF,$ then  $\CF\rank(S)\leq \e(S).$
\end{corollary}

The following result is easy to prove.
\begin{lemma}\label{lemma12}
	Let $S\in \CF.$ Then $\CF\rank(S)=0$ if and only if $S=\min(\CF).$
\end{lemma}

\begin{lemma}\label{lemma13}
	If $A$ is an $\CF$-set, $S=\CF[A],$ $x\in \msg(S)$ and $S\backslash \{x\}\in \CF,$ then $x\in A.$
\end{lemma}
\begin{proof}
	If $x\notin A,$ then $A\subseteq S\backslash \{x\}$ and $S\backslash \{x\}\in \CF.$ Then by applying Proposition \ref{proposition8}, we have that $S=\CF[A]\subseteq S\backslash \{x\},$ which is absurd.
\end{proof}
By applying Lemmas \ref{lemma12} and \ref{lemma13}, it is straightforward to check the next result.

\begin{proposition}\label{proposition14}
	Let $S\in \CF.$ Then  $\CF\rank(S)=1$ if and only if $S\neq \min(\CF)$ and $S=\CF[\{\mu(\CF,S)\}].$
\end{proposition}

\section{The semi-covariety $\theta(\Delta)$}
Along this section $\Delta$ denotes a numerical semigroup and $\theta(\Delta)=\{S\mid S \mbox{ is a nu-}\\ \mbox{merical semigroup and } \Delta \subseteq S\}.$

The next result is well known and  easy to prove. 

\begin{lemma}\label{lemma15}
	Let $S$ and $T$ be numerical semigroups and $x\in S.$ Then the following conditions are verified:
	\begin{enumerate}
		\item[1)] $S\cap T$ is a numerical semigroup and $\F(S\cap T)=\max \{\F(S),\F(T)\}.$
		\item[2)] $S\backslash \{x\}$ is a numerical semigroup if and only if $x\in \msg(S).$
		\item[3)] If $S\subsetneq T,$ then $\min(T\backslash S)\in \msg(T).$
	\end{enumerate}
	
\end{lemma}
 
 \begin{proposition}\label{proposition16}
 	Under the standing hypothesis and notation, $\theta(\Delta)$ is a semi-covariety and $\Delta$ is its minimum.
 \end{proposition}
 \begin{proof}
 	It is clear that $\Delta$ is the minimum of $\theta(\Delta).$ By applying {\it (1)} of Lemma \ref{lemma15}, it is easy to check that if $\{S,T\}\subseteq \theta(\Delta),$ then $S\cap T \in \theta(\Delta).$ Let us see that if $S\in \theta(\Delta)$ and $S\neq \Delta,$ then there is $x\in \msg(S)$ such that $S\backslash \{x\}\in \theta(\Delta).$ Indeed, if $S\neq \Delta$, there is $x=\min(S\backslash \Delta).$ Using  {\it (3)} of Lemma \ref{lemma15}, we deduce that $\min(S\backslash \Delta) \in \msg(S)$ and $S\backslash \{x\}\in \theta(\Delta).$
 \end{proof}
 
 It is easy to  prove the following result.
 
 \begin{proposition}\label{proposition17}
 If $S\in \theta(\Delta)$ and $S\neq \Delta,$ then $\mu(\theta(\Delta),S)=\min(S\backslash \Delta).$
 \end{proposition}
 
 As a consequence of Propositions \ref{proposition6} and \ref{proposition17}, we can state the following.  
 
 \begin{corollary}\label{corollary18}
 If $S\in \theta(\Delta),$ then  the set formed by the children of $S$ in the tree $\G(\theta(\Delta))$ is  
 $
 \{S\cup \{x\}\mid x\in \SG(S)\mbox{ and } \{s\in S\mid s<x\}\subseteq \Delta\}.
 $
 \end{corollary}

 Let $S$ be a numerical semigroup and $n\in S\backslash \{0\}$. The
 Apéry set of $n$ in $S$ (see \cite{apery}) is
 $\Ap(S,n)=\{s\in S\mid s-n \notin S\}$.
 
Given a nonzero integer $n$ and two integers $a$ and $b,$ we write $a \equiv b\,\mod\,  n$ to
denote that $n$ divides $a-b.$ We denote by $a\, \mod\, n$ the remainder of the division
of $a$ by $n.$

  The following result is deduced from \cite[Lemma 2.4]{libro}.
 
 \begin{lemma}\label{lemma10}
 	If $S$ is a  numerical semigroup and $n\in S\backslash \{0\},$ then $\Ap(S,n)$ is a set with cardinality $n.$ Moreover, $\Ap(S,n)=\{0=w(0),w(1), \dots, w(n-1)\}$, where $w(i)$ is the least
 	element in $S$ congruent with $i$ modulo $n$, for all $i\in
 	\{0,\dots, n-1\}.$
 \end{lemma}
 \begin{remark}\label{nota20} Notice that:
 	\begin{enumerate}
 		\item In \cite[Remark 1]{covariedades} it is shown that if $S$ is a numerical semigroup and we know $\Ap(S,n)$ for some $n\in S\backslash \{0\},$ then we easily compute $\SG(S).$ 
 			\item In \cite[Remark 2]{covariedades} it is displayed that if $S$ is a numerical semigroup and we know $\Ap(S,n)$ for some $n\in S\backslash \{0\},$ then we easily calculate $\Ap(S\cup \{x\},n)$ for all  $x\in\SG(S).$
 	\end{enumerate}
 	
 \end{remark}

We have now all the ingredients needed to compute the semi-covariety $\theta(\Delta).$
 
 	\begin{algorithm}\label{algorithm21}
 		\caption{Computation of $\theta(\Delta)$}\mbox{}\par
 	\noindent\textsc{Input}: A numerical semigroup  $\Delta.$ \\
 	\noindent\textsc{Output}: $\theta(\Delta).$
 	\begin{algorithmic}
 		\item[(1)] $\theta(\Delta)=\{\Delta\},$ $B=\{\Delta\}$ and $\Ap(\Delta, \m(\Delta)).$ 
 		\item[(2)] For every $S \in B,$ compute $$H(S)=\{x\in \SG(S)\mid  \{s\in S\mid s<x\}\subseteq \Delta\}.$$
 		\item[(3)] If $\displaystyle\bigcup_{S\in B}H(S)=\emptyset,$ then return $\theta(\Delta).$
 		\item[(4)]  $C=\displaystyle\bigcup_{S\in B}\{S\cup \{x\}\mid x\in H(S)\}.$ 	
 		\item[(5)]  $\theta(\Delta):= \theta(\Delta)\cup C,$ $B:=C.$
 		\item[(6)] For all $S\in B,$ compute $\Ap(S,\m(\Delta))$  and   go to Step $(2).$ 	
 	\end{algorithmic}
 \end{algorithm}
 
 We are going to illustrate the usage of the previous Algorithm   with\, an example.
 
 \begin{example}
 	Let $\Delta =\langle 3,7,8\rangle= \{0,3,6,\rightarrow\}.$ 
 	\begin{itemize}
 		\item $\theta(\Delta) =\{\Delta\}$, $B=\{\Delta\}$ and  $\Ap(\Delta,3)=\{0,7,8\}.$
 		\item $H(\Delta)=\{4,5\}.$ 
 		\item $C=\{ \Delta \cup \{4\}, \Delta \cup \{5\}\}.$
 		\item $\theta(\Delta)=\{\Delta, \Delta \cup \{4\},\Delta \cup \{5\}\} $ and  $B=\{\Delta \cup \{4\},\Delta \cup \{5\}\}.$ 
 		\item $\Ap(\Delta \cup \{4\},3)=\{0,4,8\}$ and $\Ap(\Delta \cup \{5\},3)=\{0,5,7\}.$
 		\item $H(\Delta \cup \{4\})=\emptyset $ and $H(\Delta \cup \{5\})=\{4\}.$
 		\item  $C=\{ \Delta \cup \{4,5\}\}.$
 		\item $\theta(\Delta)=\{\Delta, \Delta \cup \{4\},\Delta \cup \{5\}, \Delta \cup \{4,5\} \} $ and  $B=\{\Delta \cup \{4,5\}\}.$ 
 		\item $\Ap(\Delta \cup \{4,5\},3)=\{0,4,5\}.$ 
 		\item $H(\Delta \cup \{4,5\})=\{2\}.$
 		\item  $C=\{ \Delta \cup \{2,4,5\}\}.$
 		\item $\theta(\Delta)=\{\Delta, \Delta \cup \{4\},\Delta \cup \{5\}, \Delta \cup \{4,5\}, \Delta \cup \{2,4,5\} \} $ and  $B=\{\Delta \cup \{2,4,5\}\}.$ 
 		\item $\Ap(\Delta \cup \{2,4,5\},3)=\{0,2,4\}.$ 
 		\item $H(\Delta \cup \{2,4,5\})=\{1\}.$
 		 \item $C=\{ \N\}.$
 		 \item $\theta(\Delta)=\{\Delta, \Delta \cup \{4\},\Delta \cup \{5\}, \Delta \cup \{4,5\}, \Delta \cup \{2,4,5\}, \N \} $ and  $B=\{\N\}.$ 
 		\item $\Ap(\N,3)=\{0,1,2\}.$ 
 		\item $H(\N)=\emptyset.$ 		
 		\item Algorithm 2, returns $\theta(\Delta)=\{\Delta, \Delta \cup \{4\}, \Delta \cup \{5\},\Delta \cup \{4,5\},\Delta \cup \{2,4,5\}, \N\}.$
 		
 	\end{itemize}
 \end{example}
 
 The following result has a trivial proof.
 
 \begin{proposition}\label{proposition23}Under the standing  notation,
 	$X$ is a $\theta(\Delta)$-set if and only if $X\subseteq \N\backslash \Delta.$
 \end{proposition}
 \begin{proposition}\label{proposition24}
 	If $S\in \theta(\Delta),$ then $A=\{x\in \msg(S)\mid x\notin \Delta\}$ is the unique $\theta(\Delta)$-minimal system of generators of $S.$
 \end{proposition}
 \begin{proof}
 	By Proposition \ref{proposition10}, we deduce that $A$ is a $\theta(\Delta)$-system of generators of $S.$ 
 	To conclude the proof we will see that if  $B$ is a $\theta(\Delta)$-system of generators of $S,$ then $A\subseteq B.$ 
 	In fact, if  $A  \nsubseteq B,$	 there is $a\in A \backslash B.$ It is clear that $S\backslash \{a\}\in \theta(\Delta)$ and $B\subseteq S\backslash \{a\}.$ Therefore, by applying Proposition \ref{proposition8}, $S=\theta(\Delta)[B]\subseteq S\backslash \{a\},$ which is impossible. 
 \end{proof}

If $S\in \theta(\Delta),$ we denote by $\theta(\Delta)\msg(S)$ the unique $\theta(\Delta)$-minimal system of generators of $S.$ Notice that the cardinality of $\theta(\Delta)\msg(S)$ is equal to $\theta(\Delta)\rank(S).$

\begin{example}
	It is clear that $S=\langle 3,4,5 \rangle \in \theta(\langle 5,7\rangle).$ By applying Proposition \ref{proposition24}, we have $\theta(\langle 5,7\rangle)\msg(S)=\{3,4\}.$ Thus, $\theta(\langle 5,7 \rangle)\rank(S)=2.$
\end{example}

If $A$ and $B$ are non-empty subsets of $\Z$, we write $A+B=\{a+b \mid a\in A, b\in B\}.$ 

\begin{proposition}\label{proposition26}
	If $\emptyset \neq X \subseteq \N\backslash \Delta,$ then $\theta(\Delta)[X]=\Delta+\langle X \rangle.$
\end{proposition}
\begin{proof}
	It is clear that $\Delta+\langle X \rangle$ is a numerical semigroup containing  $\Delta.$ Thus, by Proposition \ref{proposition8}, we have $\theta(\Delta)[X]\subseteq \Delta + \langle X \rangle.$ 
	
	If $S\in \theta(\Delta)$ and $X\subseteq S,$ then $\Delta\subseteq S$ and $\langle X \rangle\subseteq S.$ So, $\Delta+\langle X\rangle\subseteq S$.  By applying again Proposition \ref{proposition8}, we have $\Delta+\langle X \rangle \subseteq \theta(\Delta)[X].$
\end{proof}

\begin{example}\label{example27}
	Let $\Delta=\langle 5,7,9 \rangle$ and $X=\{4,6\}.$ Then by Proposition \ref{proposition26}, we know that $\theta(\Delta)[X]=\langle 5,7,9\rangle + \langle 4,6 \rangle=\langle 4,5,6,7,9 \rangle=\langle 4,5,6,7 \rangle.$
\end{example}
The next result follows easily from Propositions \ref{proposition23} and \ref{proposition26}. 

\begin{proposition}\label{proposition28}
	Under the standing notation, $\{S\in \theta(\Delta)\mid \theta(\Delta)\rank(S)=1\}=\{\theta(\Delta)[\{x\}]\mid x\in \N \backslash \Delta\}.$
	Moreover, if $\{x,y\}\subseteq \N\backslash \Delta,$ then $\theta(\Delta[\{x\}])=\theta(\Delta[\{y\}])$ if and only if $x=y.$
\end{proposition} 

As an immediate consequence of Proposition \ref{proposition28}, we can state the following. 

\begin{corollary}\label{corollary29}
	The cardinality of $\{S\in \theta(\Delta)\mid \theta(\Delta)\rank(S)=1\}$ is equal to $\g(\Delta).$
\end{corollary}
\begin{example}\label{example30}
	If $\Delta=\langle 5,7,9\rangle=\{0,5,7,9,10,12,14, \rightarrow\},$ then $\N\backslash \Delta=\{1,2,3,4,6,\\ 8,11,13\}$ and $\g(\Delta)=8.$ Then, by applying Propositions \ref{proposition26} and \ref{proposition28} we have that $\{S\in \theta(\Delta)\mid \theta(\Delta)\rank(S)=1\}=\{\theta(\Delta)[\{1\}]=\N, \theta(\Delta)[\{2\}]=\langle 2,5 \rangle,  \theta(\Delta)[\{3\}]=\langle 3,5,7 \rangle, \theta(\Delta)[\{4\}]=\langle 4,5,7 \rangle, \theta(\Delta)[\{6\}]=\langle 5,6,7,9 \rangle, \theta(\Delta)[\{8\}]=\langle 5,7,8,9 \rangle, \theta(\Delta)[\{11\}]=\langle 5,7,9,11 \rangle$ and $\theta(\Delta)[\{13\}]=\langle 5,7,9,13 \rangle.$
\end{example}
\section{The semi-covariety $\CC(\F)$}

Throughout  this section $F$ will denote a positive integer and $\CC(F)=\{S\mid S \mbox{ is a Coe-semigroup and }\F(S)=F\}.$ The following result is \cite[Proposition 1]{coe}.

\begin{proposition}\label{proposition31}
	If $S$ is a Coe-semigroup and $S\neq \N,$ then $\m(S)$ is even and $\F(S)$ is odd.
\end{proposition}

The proof of the following result is trivial.

\begin{lemma}\label{lemma32} With the previous notation, 
	\begin{enumerate}
		\item[1)] If $F$ is odd, then $\Delta(F)=\{0,F+1,\rightarrow\}$ is the minimum of $\CC(F).$
		\item[2)] $\CC(F)\neq \emptyset$ if and only if $F$ is odd.
	\end{enumerate}
\end{lemma}
In \cite[Proposition 28]{coe} appears the following lemma.
\begin{lemma}\label{lemma33}
	The intersection of finitely many Coe-semigroups is a Coe-semigroup.	
\end{lemma}
\begin{proposition}\label{proposition34}
	If $F$ is odd, then $\CC(F)$ is a semi-covariety. Moreover, $\Delta(F)$ is its minimum.
\end{proposition}
\begin{proof}
	By Lemma \ref{lemma32}, we know that $\Delta(F)$ is the minimum of $\CC(F),$ and if $\{S,T\}\subseteq \CC(F),$ by Lemma \ref{lemma33}, we have $S\cap T\in \CC(F).$ To conclude the proof, we will see that if $S\in \CC(F)$ and $S\neq \Delta(F),$ then there is $x\in \msg(S)$ such that $S\backslash \{x\}\in \CC(F).$ In fact, if $S\neq \Delta(F),$ then $A=\{\msg(S)\mid x<F\}\neq \emptyset.$ We distinguish two
	cases.
	\begin{enumerate}
		\item[1)] If there is $x\in A$ such that $x$ is odd, then it is clear that $S\backslash \{x\}\in \CC(F).$
		\item[2)] If all the elements in $A$ are even, so it is clear that $S\backslash \{\m(S)\}\in \CC(F).$
	\end{enumerate}
\end{proof}
We define the graph $\G(\CC(F))$ as the graph whose vertices are the elements of $\CC(F)$ 
 and $(S, T)\in \CC(F)\times \CC(F)$  is an edge if and only if $T=S\backslash \{\mu(\CC(F),S)\}.$
 
 By applying Propositions \ref{proposition4}, \ref{proposition6} and \ref{proposition34}, we obtain the next result.
 
 \begin{proposition}\label{proposition35} If $F$ is odd, then  $\G(\CC(F))$ is a tree with root $\Delta(F).$ Moreover, if $S\in \CC(F),$ then the set formed by the children of $S$ in the tree  $\G(\CC(F))$ is $\{S\cup \{x\}\mid x\in \SG(S), S\cup \{x\}\in \CC(F) \mbox{ and }\mu(\CC(F),S\cup \{x\})=x\}.$
 \end{proposition}

In \cite[Proposition 2]{coe} appears the following result. It allows us to check wether or not a numerical semigroup is a Coe-semigroup.

\begin{proposition}\label{proposition36}
	Let $S$ be a numerical semigroup. The following conditions 
	are equivalent. 
	\begin{enumerate}
		\item[1)] $S$ is a Coe-semigroup.
		\item[2)] If $x\in \msg(S)$ and $x$ is odd, then $\{x-1,x,x+1\}\subseteq S.$ 
	\end{enumerate}
\end{proposition}

With this last proposition we already have all the ingredients to
show an algorithm to compute the semi-covariety $\CC(F).$

\begin{algorithm}\label{algorithm37}\mbox{}\par
	\caption{Computation of $\CC(F)$}
	\noindent\textsc{Input}:  A positive odd integer $F$.   \par
	\noindent\textsc{Output}:  $\CC(F).$
	\begin{algorithmic}
		\item[(1)] $\CC(F)=\{\Delta(F)\},$ $B=\{\Delta(F)\}$ and $\Ap(\Delta(F), F+1)=\{0,F+2,\cdots, 2F+1\}.$ 
		\item[(2)] For every $S \in B,$ compute $$H(S)=\{x\in \SG(S)\mid S\cup \{x\} \in \CC(F) \mbox{ and } \{\mu(\CC(F),S\cup \{x\})=x\}.$$
		\item[(3)] If $\displaystyle\bigcup_{S\in B}H(S)=\emptyset,$ then return $\CC(F).$
		\item[(4)]  $C=\displaystyle\bigcup_{S\in B}\{S\cup \{x\}\mid x\in H(S)\}.$ 	
		\item[(5)]  $\CC(F):=\CC(F)\cup C$ and $B:=C.$
		\item[(6)] Compute $\Ap(S, F+1)$ for all $S\in B$ and   go to Step $(2).$ 	
	\end{algorithmic}
\end{algorithm}

\begin{remark}\label{nota38}
	\begin{enumerate}
		\item Observe that as a consequence of Proposition \ref{proposition36}, we have that if $S$ is a numerical semigroup with $\F(S)$ odd, then $S$ is a Coe-semigroup if and only if  $\{x-1,x,x+1\}\subseteq S$ for all $x\in \msg(S)$ such that $x$ is odd and $x<\F(S).$
		\item It is clear that if $S$ is a numerical semigroup with Frobenius number $F$, then $\{a\in \msg(S)\mid a<F\}=\{w\in \Ap(S,F+1)\mid w<F \mbox{ and }w-w'\notin \Ap(S,F+1) \mbox{ for all }w'\in \Ap(S,F+1)\backslash \{0,w\}\}.$
	\end{enumerate}
	
\end{remark}

\begin{example}\label{example39}
	We are going to compute $\CC(7)$ by using Algorithm 3.
	\begin{itemize}
		\item $\CC(7))=\{\Delta(7)\},$ $B=\{\Delta(7)\}$ and $\Ap(\Delta(7), 8)=\{0,9,10,11,12,13,14,15\}.$ 
		\item $H(\Delta(7))=\{4,6\}.$
		\item $C=\{\Delta(7)\cup \{4\}, \Delta(7)\cup \{6\}\}.$
			\item $\CC(7)=\{\Delta(7), \Delta(7)\cup \{4\}, \Delta(7)\cup \{6\}\}$ and $B=\{\Delta(7)\cup \{4\}, \Delta(7)\cup \{6\}\}.$ 
		\item $\Ap(\Delta(7)\cup \{4\}, 8)=\{0,4,9,10,11,13,14,15\}$ and  $\Ap(\Delta(7)\cup \{6\}, 8)=\{0,6,9,10,11,12,13,15\}.$
			\item $H(\Delta(7)\cup \{4\})=\emptyset $ and 	 $H(\Delta(7)\cup \{6\})=\{4\}.$
			\item $C=\{\Delta(7)\cup \{4,6\}\}.$
				\item $\CC(7)=\{\Delta(7), \Delta(7)\cup \{4\}, \Delta(7)\cup \{6\}, \Delta(7)\cup \{4,6\}\}$ and  $B=\{\Delta(7)\cup \{4,6\}\}.$ 
			\item $\Ap(\Delta(7)\cup \{4,6\}, 8)=\{0,4,6,9,10,11,13,15\}.$
				\item $H(\Delta(7)\cup \{4,6\})=\{5,2\}.$
					\item $C=\{\Delta(7)\cup \{2,4,6\}, \Delta(7)\cup \{4,5,6\}\}.$
				\item $\CC(7)=\{\Delta(7), \Delta(7)\cup \{4\}, \Delta(7)\cup \{6\}, \Delta(7)\cup \{4,6\}, \Delta(7)\cup \{2,4,6\}, \Delta(7)\cup \{4,5,6\} \}$ and  $B=\{\Delta(7)\cup \{2,4,6\}, \Delta(7)\cup \{4,5,6\}\}.$ 
				\item $\Ap(\Delta(7)\cup \{2,4,6\}, 8)=\{0,2,4,6,9,11,13,15\}$ and $\Ap(\Delta(7)\cup \{4,5,6\}, 8)=\{0,4,5,6,9,10,11,15\}.$ 
					\item $H(\Delta(7)\cup \{2,4,6\})=\emptyset$ and $H(\Delta(7)\cup \{4,5,6\})=\emptyset.$
					\item Algorithm 3 returns  $\CC(7)=\{\Delta(7), \Delta(7)\cup \{4\}, \Delta(7)\cup \{6\}, \Delta(7)\cup \{4,6\}, \Delta(7)\cup \{2,4,6\}, \Delta(7)\cup \{4,5,6\} \}.$
	\end{itemize}
\end{example}
If $x$ is an  positive odd integer,  denote by $\c(x)=\{x-1,x+1\}.$ If $X\subseteq \N\backslash \{0\}, $ then $\C(X)=X\cup \left ( \displaystyle \bigcup_{ \begin{array}{c}
		x\in X\\
		x \mbox{ impar}\\
\end{array}}
\c(x)
\right).$
\begin{proposition}\label{proposition40}
	If $F$ is  odd and $X\subseteq \{1,\cdots, F-1\},$ then $X$ is a $\CC(F)$-set if and only if $F\notin \langle \C(X)\rangle.$ 
\end{proposition}
\begin{proof}
{\it Necessity.} If $X$ is a $\CC(F)$-set, then $\CC(F)[X]\in \CC(F).$ Therefore, $\C(X)\subseteq \CC(F)[X]$ and so  $\langle \C(X)\rangle \subseteq  \CC(F)[X].$ Consequently, $F\notin \langle \C(X)\rangle.$

{\it Sufficiency.}	By applying Proposition \ref{proposition36}, we easily deduce that $S=\langle \C(X)\rangle \cup \{F+1,\rightarrow\}\in \CC(F).$ As $X\cap \Delta(F)=\emptyset$ and $X\subseteq S,$ then $X$ is a $\CC(F)$-set.
\end{proof}

\begin{proposition}\label{proposition41}
	Let $F$ be a positive odd  integer and $X$ be a $\CC(F)$-set, then $\CC(F)[X]=\langle \C(X)\rangle \cup \{F+1,\rightarrow\}.$
\end{proposition}
\begin{proof}
	Using Propositions \ref{proposition36} and \ref{proposition40}, it follows easily that $S=\langle \C(X)\rangle \cup \{F+1,\rightarrow\}\in \CC(F).$ It is clear that every element of $\CC(F)$ that contains $X$ must contain $\C(X).$ Consequently, $S$ is the smallest element of $\CC(F)$ containing $X.$ By Proposition \ref{proposition8}, we have $\CC(F)[X]=S.$
\end{proof}
\begin{proposition}\label{proposition42}
	If $F$ is  odd and $S\in \CC(F),$ then $A=\{x\in \msg(S)\mid x<F \mbox{ and } x \mbox{ is  odd } \}\cup \{x\in \msg(S)\mid x<F,  x \mbox{ is even and } \{x-1,x+1 \}\cap \msg(S)=\emptyset\}$ is the unique $\CC(F)$-minimal system of generators of $S.$
\end{proposition}
\begin{proof}
	By Proposition \ref{proposition41}, we know that $\CC(F)[A]=\langle \C(A) \rangle \cup \{F+1,\rightarrow\}.$ It is clear that $\{x\in \msg(S)\mid x<F\}\subseteq \C(A)$ and so, $S\subseteq \CC(F)[A].$ As $S\in \CC(F)$ and $A \subseteq S,$ then by Proposition \ref{proposition8}, we have that $\CC(F)[A]\subseteq S.$ Hence, $\CC(F)[A]=S.$
	
	In order to conclude the proof, it suffices to prove that if $B$ is a $\CC(F)$-set and $\CC(F)[B]=S,$ then $A\subseteq B.$ In fact, if $A \not\subseteq B, $ then there is $a\in A\backslash B.$ Therefore, we easily obtain that $S\backslash\{a\}\in \CC(F)$ and $B\subseteq S\backslash \{a\}.$ By applying Proposition \ref{proposition8}, we have that $S=\CC(F)[B]\subseteq S\backslash \{a\},$ which is absurd.
	\end{proof}
	\begin{example}\label{example43}
		By applying Proposition \ref{proposition36}, we have that $S=\langle 4,5,6\rangle\in \CC(7)$ and by applying Proposition \ref{proposition42}, we have that $\{5\}$ is the $\CC(7)$-minimal system of generators of $S.$		
	\end{example}

\begin{remark}
	\label{remark43a}
We now focus on the study of elements from $\CC(F)$ with $\CC(F)$-$\rank$ $1.$ Notice that if $x$ is a positive even  integer and $x<F,$ then by Proposition \ref{proposition40}, we know that $\{x\}$ is a $\CC(F)$-set, and by Proposition \ref{proposition41} we obtain that $\CC(F)[\{x\}]=\langle x \rangle \cup \{F+1,\rightarrow\}.$

Notice also that if $x$ is a  positive odd positive integer less than $F,$ then by Proposition \ref{proposition40}, $\{x\}$ is a $\CC(F)$-set if and only if $F\notin \langle x-1,x,x+1\rangle$ and by Proposition \ref{proposition41} we have $\CC(F)[\{x\}]=\langle x-1,x,x+1 \rangle \cup \{F+1,\rightarrow\}.$

\end{remark}

	If $q$ is a rational number,  $\lfloor q \rfloor=\max \{z\in \Z\mid z\leq q\}.$ 
	
	The following results is deduced from \cite[Corollary 2]{pacific}.

\begin{lemma}\label{lemma44}
	Let $x$ be an odd integer such that $3\leq x\leq F-2.$ Then $F\notin \langle x-1,x,x+1\rangle$ if and only if $2\left\lfloor \frac{F}{x-1}\right\rfloor < F\, \mod\, (x-1).$
	
\end{lemma}

The next result is a consequence of Lemma \ref{lemma44} and Remark \ref{remark43a}.
\begin{proposition}\label{proposition45}
	With the above notation, we have that $S$ is an element of $\CC(F)$ with $\CC(F)$-$\rank$ 1 if and only if  one of the following conditions hods.
	\begin{enumerate}
		\item[1)] $S=\langle x \rangle \cup \{F+1,\rightarrow\}$ for some even integer $x$ such that $2\leq x \leq F-1.$
		\item[2)] $S=\langle x-1,x,x+1\rangle \cup \{F+1,\rightarrow\}$ for some odd integer $x$ such that $3\leq x \leq F-2$ and $2\left\lfloor \frac{F}{x-1}\right\rfloor < F\, \mod\, (x-1).$
	\end{enumerate}
	
\end{proposition}

\end{document}